%% file: Twisted_Embeddings_of_Tori.tex
\documentclass[12pt]{amsart}

\usepackage{amsmath}
\pdfoutput=1 
\usepackage{amsfonts}
\usepackage{amssymb}
\usepackage{amsthm}
\usepackage{thmtools}
\usepackage{url}
\usepackage{color}
\usepackage[utf8]{inputenc}
\usepackage[foot]{amsaddr}

\makeatletter
\renewcommand{\email}[2][]{%
  \ifx\emails\@empty\relax\else{\g@addto@macro\emails{,\space}}\fi%
  \@ifnotempty{#1}{\g@addto@macro\emails{\textrm{(#1)}\space}}%
  \g@addto@macro\emails{#2}%
}
\makeatother

\allowdisplaybreaks

\usepackage{upgreek}
\usepackage{subcaption}
\usepackage{tikz-cd}
\usepackage{bigdelim}
\usepackage{adjustbox}
\usepackage{multirow}
\usepackage{hhline}
\usepackage{enumerate}
\usepackage{graphicx}
\usepackage{enumitem}
\usepackage{comment}
\usepackage{amsthm,verbatim}
\usepackage{latexsym}
\usepackage{fancyhdr}
\usepackage{import}
\usepackage[capitalize]{cleveref}
\usepackage[margin=1in]{geometry}

\DeclareMathOperator*{\im}{Im}
\DeclareMathOperator*{\Area}{area}
\DeclareMathOperator*{\diam}{diam}
\DeclareMathOperator*{\ch}{ch}

\DeclareMathOperator*{\length}{length}
\DeclareMathOperator*{\Vol}{Vol}

\DeclareMathOperator*{\sys}{sys}

\DeclareMathOperator*{\dist}{distor}

\DeclareMathOperator*{\interior}{int}

\DeclareMathOperator*{\genus}{genus}

\DeclareMathOperator*{\cell}{cell}

\DeclareMathOperator*{\sign}{sign}
\DeclareMathOperator*{\exterior}{ext}

\renewcommand{\epsilon}{\varepsilon}
\renewcommand{\Lambda}{\Uplambda}

\newtheorem{theorem}{Theorem}[section]
\newtheorem{lemma}[theorem]{Lemma}

\newtheorem{question}[theorem]{Question}
\theoremstyle{definition}
\newtheorem{definition}[theorem]{Definition}
\newtheorem{remark}[theorem]{Remark}
\newtheorem{example}[theorem]{Example}

\begin{document}

\title[Twisted Embeddings of Tori]{Twisted embeddings of tori have small extrinsic systole}
\date{\today}
\author[Vasudevan]{Sahana Vasudevan}
\address[Sahana Vasudevan]{School of Mathematics, Institute for Advanced Study, Princeton, NJ 08540, USA; Department of Mathematics, Princeton University, Princeton, NJ 08540, USA}
\email{sahanav@ias.edu}

\fancyhead[CE]{Vasudevan}

\begin{abstract} We prove a type of systolic inequality for embeddings of $T^2$ in $\mathbb{R}^3$. In particular, a highly twisted $T^2$ embedded in $\mathbb{R}^3$ must contain a non-contractible loop of small $\mathbb{R}^3$-diameter. 
\end{abstract} 

\maketitle

\section{Introduction}

The systole of a Riemannian manifold is the length of its shortest non-contractible loop. Systolic inequalities compare the systole to other geometric invariants, like volume. They have been studied since the 1940s, when Loewner proved a systolic inequality for two-dimensional tori. In this paper, we introduce the extrinsic systole, a generalization of the systole that applies to an embedded submanifold of a Riemannian manifold. We prove an extrinsic systolic inequality for two-dimensional tori embedded in $\mathbb{R}^3$. This inequality is in part motivated by questions about the distortion of knots, a geometric knot invariant introduced by Gromov in the 1980s.

We start with Loewner's systolic inequality. Let $T$ be a two-dimensional torus equipped with a Riemannian metric. 

\begin{theorem}[Loewner, 1949] \label{Loewner's ineq} There exists a non-contractible closed curve $\gamma$ on $T$ such that $$\length(\gamma)\leq 2^{1/2}3^{-1/4}\Area(T)^{1/2}.$$
\end{theorem}

\cref{Loewner's ineq} gives an upper bound on the systole of $T$, henceforth denoted $\sys(T)$. The systole may alternatively be described as the diameter of the smallest metric ball on $T$ containing a non-contractible closed curve.

In this paper, we consider two-dimensional tori $T$ embedded in $\mathbb{R}^3$. Motivated by Loewner's inequality, we investigate global geometric properties of the embedding. Our main result is an upper bound on the diameter of the smallest ball in $\mathbb{R}^3$ that contains a non-contractible curve on $T$. The upper bound depends only on the metric on $T$ and the topology of the embedding, and in some cases, is much smaller than $\sys(T)$. 

In order to state our main result, let us describe the topology of the embedding in more detail. The complement of the torus, $\mathbb{R}^3-T$, has two connected components. Let $u\in H_1(T,\mathbb{Z})$ be a nonzero primitive homology class that is trivial in one of the connected components of $\mathbb{R}^3-T$. Let $v\in H_1(T,\mathbb{Z})$ be a nonzero primitive homology class that is trivial in the other connected component of $\mathbb{R}^3-T$.

\begin{definition} The homological systole with respect to $u$ is $$\textstyle\sys_u(T)=\displaystyle\inf_{[\gamma]=u}\length(\gamma),$$ where $\gamma$ ranges over closed curves on $T$ representing $u$.
\end{definition}

Our main result is the following inequality:

\begin{theorem}\label{main embedding theorem} There exists a closed curve $\gamma$ on $T$ representing $v$ such that $$\textstyle\diam_{\mathbb{R}^3} \gamma\lesssim \sys_{u}(T)^{-1/3}\Area(T)^{2/3}+\sys_{u}(T)^{-1}\Area(T).$$ 
\end{theorem}

\begin{remark} Here, the symbol $\lesssim$ means that the left-hand side is less than the right-hand side, up to multiplication by a universal constant. Henceforth, the underlying constant in the symbol $\lesssim$ will always be universal.
\end{remark}

\begin{remark} Note that the inequality is invariant under scaling of $T$. Furthermore, $T$ need not be unknottedly embedded in $\mathbb{R}^3$. In any case, there are always two homology classes on $T$ that are trivial in the two components of the complement. Note also that the hypotheses regarding $u$ and $v$ are symmetric.
\end{remark}

\cref{main embedding theorem} is most interesting when $\sys_u(T)$ is large compared to $\Area(T)$. In this case, the conclusion of the theorem is that a small ball in $\mathbb{R}^3$ contains a non-contractible curve on $T$, even if $\sys(T)$ is not necessarily small. This motivates the following definition:

\begin{definition} The extrinsic systole of $T$, denoted $\sys_{\mathbb{R}^3}(T)$, is the diameter of the smallest ball in $\mathbb{R}^3$ that contains a non-contractible curve on $T$.
\end{definition}

\begin{remark} The definition of extrinsic systole extends more generally to submanifolds $M$ of a Riemannian manifold $N$. Assuming $\pi_1(M)$ is nontrivial, we define $\sys_N(M)$ to be the smallest metric ball on $N$ containing a non-contractible closed curve on $M$. When $M=N$, this definition recovers the (intrinsic) systole. 
\end{remark}

Then \cref{main embedding theorem} is an extrinsic systolic inequality: an upper bound on $\sys_{\mathbb{R}^3}(T)$ in terms of intrinsic geometric quantities on $T$. The upper bound can be much less than $\sys(T)$, as seen in the following special case of the theorem.

\begin{example} \label{twisted cylinder}
Let $T=\mathbb{R}^2/\mathbb{Z}^2$ be the flat torus equipped with the Euclidean metric. Take any unknotted and $C^1$-isometric embedding of $T$ in $\mathbb{R}^3$, with the topology of the embedding specified as follows. Let $u$ and $v$ be homology classes in $H_1(T)$ trivial in the homology of the two components of $\mathbb{R}^3-T$, respectively. Note that closed curves on $T$ are given by points on $\mathbb{Z}^2$. Embed $T$ in $\mathbb{R}^3$ so that under this identification, 
\begin{equation*}
\begin{cases} (1,0)=u+nv\\
(0,1)=v,
\end{cases}
\end{equation*} for a positive integer $n$. Geometrically, such an embedding may be constructed by taking a standard cylinder of height $1$ and circumference $1$, and twisting one end $n$ times before gluing it to the other end, as shown in Figure 1. By \cref{main embedding theorem}, any embedding with the prescribed topology satisfies $\sys_{\mathbb{R}^3}(T)\lesssim n^{-1/3}$. On the other hand, $\sys(T)=1$.

\begin{figure}[h]
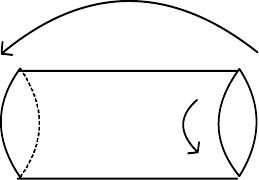
\caption{}
\end{figure}
\end{example}

\begin{remark} \cref{twisted cylinder} may be empirically tested by twisting a piece of paper, as follows. Take a $1/2$-by-$1$ strip of paper; twist one end $n$ full times and glue it to the other end. One may observe that for large $n$, the twisting process is impossible to execute unless one ``pinches'' the paper. Viewing the unglued paper as a degenerate cylinder, and the glued paper as a degenerate torus, ``pinching'' the paper is equivalent to creating small extrinsic systole. 
\end{remark}

The inequality in \cref{main embedding theorem} is sharp (up to a constant factor) as $\sys_u(T)\to \infty$. We include a sharp example for it in \cref{Appendix A1}. On the other hand, it is not clear if the bound $n^{-1/3}$ in \cref{twisted cylinder} is sharp as $n\to \infty$, over all $n$-twisted embeddings of the flat torus. The flat torus $T=\mathbb{R}^2/\mathbb{Z}^2$ admits an $n$-twisted $1$-Lipschitz map to a standard torus in $\mathbb{R}^3$ scaled by $\sim n^{-1}$. By the Nash isometric embedding theorem, we may perturb this map in an arbitrarily small neighborhood of the standard torus so that it is a $C^1$-isometry. The resulting embedding of $T$ in $\mathbb{R}^3$ has extrinsic systole $\sim n^{-1}$. It is open if there is another embedding of $\mathbb{R}^2/\mathbb{Z}^2$ with the prescribed topology whose extrinsic systole is larger. (In \cref{Appendix A3}, we describe a generalization of this question.)

\subsection{Intrinsic systolic geometry and ideas in the proof of \cref{main embedding theorem}} 
After Loewner's inequality, many other systolic inequalities and generalizations have been proved both in the two-dimensional (\cite{Pu52}, \cite{Bla61}, \cite{Bav86}, \cite{Gro96},  \cite{BS10}, \cite{BPS12}) and higher-dimensional (\cite{Gro83}, \cite{Gro99}, \cite{BK03}, \cite{BK04}, \cite{IK04}) settings. (See \cite{Kat07} for a more complete list of references.) In particular, variants of the systole (homological, stable, as well as higher-dimensional) have been studied. However, all of these geometric invariants are intrinsic invariants. Extrinsic systolic invariants have not been introduced before, and \cref{main embedding theorem} appears to be the first extrinsic systolic inequality.  

The deepest result in intrinsic systolic geometry is Gromov's systolic inequality in \cite{Gro83}, which generalizes Loewner's inequality to higher dimensional essential manifolds (including higher dimensional tori). Gromov's systolic inequality is connected to various areas of geometry and topology: isoperimetric inequalities, scalar curvature, topological dimension theory, and hyperbolic geometry; these connections give different proofs of the inequality (see \cite{Gut10}). In particular, Gromov's original proof exploits a connection to an analogue of the Euclidean isoperimetric inequality on a Banach space. Our proof of \cref{main embedding theorem} also exploits a connection to an isoperimetric inequality, but this isoperimetric inequality is on a certain type of $2$-dimensional CW-complex instead.

We now give a high-level overview of the proof of \cref{main embedding theorem} (see \cref{technical overview of proof of main theorem} for a more rigorous but more technical overview of most of the proof). Recall that we have a torus $T$ embedded in $\mathbb{R}^3$, and we wish to prove an upper bound on its extrinsic systole. The first step is to translate the theorem into a discretized variant that involves the intersection of $T$ with the skeleta of a certain CW-complex, namely, a cubical lattice in $\mathbb{R}^3$. Roughly speaking, the variant states that if $T$ intersects the $1$ and $2$-skeleta of the cubical lattice in a bounded way, and no $3$-cell contains a non-contractible curve of $T$, then $\sys_u(T)$ also has an upper bound. The variant recovers the original theorem by integrating over a family of such CW-complexes. (In \cref{reduction to CW intersection}, we state the variant precisely and deduce \cref{main embedding theorem} as a consequence.) What we gain from this translation is the ability to write bounds on the extrinsic systole as a purely topological statement. For example, if a cube ($3$-cell) of the CW-complex contains a non-contractible curve on $T$, then the diameter of a cube is an upper bound for the extrinsic systole of $T$. 

The variant is a consequence of an isoperimetric inequality on the $2$-skeleton of the cubical lattice. Intuitively, this isoperimetric inequality states that any $1$-chain on the $2$-skeleton with bounded norm may be filled by a $2$-chain on the $2$-skeleton with bounded norm; we use algebraic topology to formalize this statement and define what norm means in this context. In \cref{isoperimetric section}, we state and prove such an isoperimetric inequality. Finally, in \cref{proof of CW intersection}, we use this to prove the variant.

\subsection{Motivation from knot distortion} 

In \cite{Gro83}, Gromov introduced the notion of distortion $\dist(K)$ associated to a knot $K$. Distortion is an invariant of $K$ that measures the infimum of the bi-Lipschitz constant of embeddings of $K$ in $\mathbb{R}^3$, over all such embeddings. Although distortion has an elementary definition, the question of estimating $\dist(K)$ for a knot $K$ has been very difficult, even in simple cases.

In \cite{Par11}, Pardon proved that torus knots $T_{p,q}$ satisfy that $\dist(T_{p,q})$ is at least $\min\{p,q\}$ up to a constant factor. At a high-level, Pardon's argument also involves reducing the distortion inequality to a discretized variant, involving the intersection of a torus containing  $T_{p,q}$ with the skeleta of a certain CW-complex in $\mathbb{R}^3$. The variant is then a consequence of a purely topological statement concerning the intersection of tori with CW-complexes.

A modification of Pardon's argument gives another extrinsic systolic inequality for two-dimensional tori in $\mathbb{R}^3$ (see \cref{Appendix A2} for a precise statement and proof). The inequality reduces to a discretized variant that is a modification of the one above, and a consequence of the same topological statement as above. Roughly speaking, the inequality says that if $\sys_u(T)$ and $\sys_v(T)$ are both large compared to $\Area(T)$, then $\sys_{\mathbb{R}^3}(T)$ is small. It is still open whether $\dist(T_{2,q})\to \infty$ as $q\to \infty$; if we expect this statement to be true, it is natural to ask whether just $\sys_u(T)$ being large compared to $\Area(T)$ implies that $\sys_{\mathbb{R}^3}(T)$ is small. This was in part our original motivation to prove \cref{main embedding theorem}. 

In \cite{Vas25}, we prove new lower bounds for the distortion of certain knots, and upper bounds for the extrinsic systole of Seifert surfaces associated to these knots. Again, the technique involves reducing these inequalities to discretized variants involving the intersection of the Seifert surface with a CW-complex in $\mathbb{R}^3$. In this context too, both the distortion inequality and extrinsic systolic inequality are consequences of the same underlying topological statement about how surfaces intersect CW-complexes.

\subsection*{Acknowledgments} I thank Sasha Berdnikov, Larry Guth and Shmuel Weinberger for many conversations related to this paper. I have been supported by NSF DMS-2202831 and the Friends of the Institute for Advanced Study.

\section{A variant of \cref{main embedding theorem}}\label{reduction to CW intersection}

In this section, we state a variant of \cref{main embedding theorem}, and prove that the theorem follows as a consequence of the variant. 

Let $X$ be the infinite CW $3$-complex given by a cubical lattice in $\mathbb{R}^3$. Let $X^1$ and $X^2$ be the $1$-skeleton and $2$-skeleton of $X$, respectively. All homology will be with coefficients in $\mathbb{Z}$.

\begin{theorem}\label{main embedding theorem CW version} Let $T\subset X\simeq \mathbb{R}^3$ be an embedded torus, in general position with the CW structure of $X$. Let $u$ and $v$ be primitive homology classes in $H_1(T)$ that are trivial in the two components of $\mathbb{R}^3-T$, respectively. Suppose $$\textstyle\sys_u(T)\geq 20\cdot(|X^1\cap T|+1)\cdot\length(X^2\cap T).$$ Then there exists a $3$-cell of $X$ containing a representative of $v$.
\end{theorem}

\subsection{Proof of \cref{main embedding theorem}}

Perturb $T$ infinitesimally so that it is smooth. Scale $T$ so that $\Area(T)=1$, noting that the inequality in \cref{main embedding theorem} is invariant under scaling. We wish to show that there exists $\gamma$, representing the homology class $v$, such that $$\textstyle\diam_{\mathbb{R}^3}\gamma\lesssim \sys_u(T)^{-1/3}+\sys_u(T)^{-1}.$$

Let $$m=\frac{(\sys_u(T)^{-1/3}+\sys_u(T)^{-1})^{-1}}{1000}.$$ Let $X$ be the $m^{-1}$-side length cubical lattice in $\mathbb{R}^3$. 

\begin{lemma}\label{translation} There exists a translation of $T$, denoted $T+s$, in general position with the CW structure on $X$, satisfying $$\length(X^2\cap (T+s))\leq 9 m$$ and $$|X^1\cap (T+s)|\leq 9 m^2.$$
\end{lemma}

\begin{proof} Note that $X^2\cap (T+s)$ and $X^1\cap (T+s)$ only depend on the image of $s$ under the map $$\mathbb{R}^3\to \mathbb{R}^3/m^{-1}\mathbb{Z}^3.$$ We claim, 

\begin{equation}\label{eq 1}\int_{s\in [0,m^{-1}]^3}\length(X^2\cap (T+s))ds\leq 3m^{-2}.
\end{equation}

To see this, note that $X^2$ is the union of parallel planes in three directions. Denote by $X^2_x$, $X^2_y$ and $X^2_z$ the union of the planes in $X$ perpendicular to the $x$ direction, $y$ direction and $z$ direction, respectively. Since $\Area(T)=1$, 
$$\int_{s=(0,s_y,s_z)}^{(m^{-1},s_y,s_z)}\length{(X^2_x\cap (T+s))}ds_x\leq 1.
$$ As $T$ is translated in the $y$ or $z$ direction, $X^2_x\cap (T+s)$ does not change. Hence $$\int_{s\in[0,m^{-1}]^3}\length(X^2_x\cap (T+s))ds\leq m^{-2}.$$ Similar equalities hold for the $y$ and $z$ directions. Adding these equalities, \cref{eq 1} follows. 

Next, we claim, \begin{equation}\label{eq 2}\int_{s\in [0,m^{-1}]^3}|X^1\cap (T+s)|ds\leq 3m^{-1}.
\end{equation}

To see this, note that $X^1$ is the union of parallel lines in three directions. Denote by $X^1_x$, $X^1_y$ and $X^1_z$ the union of the lines of $X^1$ in the $x$ direction, $y$ direction and $z$ direction, respectively. Then, $$\int_{s=(s_x,0,0)}^{(s_x,m^{-1},m^{-1})} |X^1_x\cap (T+s)|ds_yds_z\leq 1.$$ As $T$ is translated in the $x$ direction, $X^1_x\cap (T+s)$ does not change. Hence $$\int_{s\in[0,m^{-1}]^3}|X^1_x\cap (T+s)|ds\leq m^{-1}.$$ Similar equalities hold for the $y$ and $z$ directions. Adding these equalities, \cref{eq 2} follows. 

By \cref{eq 1}, $$\Vol(\{s\in [0,m^{-1}]^3|\length(X^2\cap (T+s))\leq 9m\})\geq 2m^{-3}/3.$$ By \cref{eq 2}, $$\Vol(\{s\in [0,m^{-1}]^3||X^1\cap (T+s)|\leq 9m^2\})\geq 2m^{-3}/3.$$ Therefore, some $s\in [0,m^{-1}]^3$ satisfies the conditions of the lemma. 
\end{proof}

We return to the proof of \cref{main embedding theorem} assuming \cref{main embedding theorem CW version}. By \cref{translation}, after an appropriate translation of $T$, we may assume that 
\begin{equation*}\label{eq 3}\length(X\cap T)\leq 9 m
\end{equation*} and 
\begin{equation*}\label{eq 4}|X_1\cap T|\leq 9 m^2.
\end{equation*} Note that the translation does not change any of the conditions in \cref{main embedding theorem}. 

Suppose that there is no representative $\gamma$ of $v$ on $T$ with $$\textstyle\diam_{\mathbb{R}^3}\gamma\leq 3^{1/2}m^{-1}.$$ Then no $3$-cell of $X$ contains a representative of $v$. By \cref{main embedding theorem CW version}, 
\begin{align*}\textstyle\sys_u(T)&\leq 20\cdot (|X_1\cap T|+1)\cdot \length(X_2\cap T)\\&\leq 1620m^3+180m \\& \leq  \frac{(\sys_u(T)^{-1/3}+\sys_u(T)^{-1})^{-3}}{617283}+\frac{(\sys_u(T)^{-1/3}+\sys_u(T)^{-1})^{-1}}{4}\\&<\textstyle\sys_u(T)
\end{align*}
which is a contradiction. Therefore, there is a representative $\gamma$ of $v$ on $T$ with \begin{align*}\textstyle\diam_{\mathbb{R}^3}\gamma &\leq 3^{1/2}m^{-1}\\&\leq 3^{1/2}\cdot 1000\cdot \textstyle(\sys_u(T)^{-1/3}+\sys_u(T)^{-1}),
\end{align*}
completing the proof of \cref{main embedding theorem}.

\subsection{Roadmap to prove \cref{main embedding theorem CW version}} \label{technical overview of proof of main theorem} Assume the contrary, that no $3$-cell of $X$ contains a representative of $v$. We first prove that this implies that the $2$-skeleton $X^2$ contains a representative of $u$. While this representative may not be embedded, if $|X^1\cap T|$ is small, we prove that its multiplicity is bounded. So if $\length(T\cap X^2)$ is also small, then $\sys_u(T)$ must be bounded above, which gives a contradiction. 

Let us describe the first step of the proof in slightly more detail. We will prove that if no $3$-cell of $X$ contains a representative of $v$, then $X^2$ contains an embedded representative of some $au+bv$ with $a$ nonzero. Since $H_1(X^2)=0$, this curve may be filled by a $2$-chain on $X^2$. Because $u$ and $v$ are the only homology classes which are trivial in the complement of $T$, the $2$-chain must actually contain a representative of $u$. To bound the multiplicity of the resulting representative, we bound the multiplicity of the $2$-chain. To this end, in \cref{isoperimetric section}, we prove a type of isoperimetric inequality on $X^2$. Then in \cref{proof of CW intersection}, we prove \cref{main embedding theorem CW version} by implementing the strategy above. 

\section{An isoperimetric inequality on $X^2$} \label{isoperimetric section} 

The goal of this section is: given an embedded curve on $X^2$, we find a $2$-chain on $X^2$ of bounded multiplicity filling the curve. To define the notion of multiplicity precisely, we work with cellular chains. 

\subsection{Cellular chains associated to arcs and curves} Let $Y$ be a CW complex. A cellular $1$-chain is a linear combination of $1$-cells. We associate a cellular chain to an arc or curve on the $1$-skeleton of $Y$. In the following, $Y$ may be an infinite CW complex, but arcs and curves on $Y$ are assumed to be finite.

\begin{definition} Let $\sigma$ be a $1$-cell of $Y$. The quotient map $$q_{\sigma}:Y^1\to Y^1/(Y^1-\sigma)\simeq S^1$$ is formed by collapsing $Y^1-\sigma$ to a point. 
\end{definition}

\begin{definition}\label{coefficients of ch} Let $\gamma:[0,1]\to Y^1$ be an arc (or $\gamma:S^1\to Y^1$ a closed curve) in the $1$-skeleton of $Y$ with $\partial \gamma\subset Y^0$. (If $\gamma$ is a closed curve, we consider it as an arc with both endpoints the same $0$-cell.) For each $1$-cell $\sigma$, $q_\sigma\circ \gamma$ is a curve on $Y^1/(Y^1-\sigma)\simeq S^1$ with endpoints at the same point. We denote by $\pi_\sigma(\gamma)$ the element of $\pi_1(S^1)\simeq \mathbb{Z}$ determined by $q_\sigma\circ \gamma$. When $\gamma$ is a closed curve, $$\pi_\sigma(\gamma)=\deg(q_\sigma\circ \gamma).$$ In this case, $\pi_\sigma(\gamma)$ is independent of choice of basepoint on $\gamma$.
\end{definition}

\begin{definition}\label{ch} Let $\gamma$ be an arc or closed curve on $Y^1$. The cellular $1$-chain associated to $\gamma$ is $$\ch(\gamma)=\sum_{\sigma \text{ is a }1\text{-cell in } Y^1}\pi_\sigma(\gamma)\sigma.$$
\end{definition}

\subsection{Basic properties of associated chains}

Associated chains add when arcs concatenate.

\begin{lemma}\label{ch concatenation} Let $\gamma$ be a concatenation of arcs $\gamma_1$ and $\gamma_2$ in $Y^1$ with boundary in $Y^0$. Then $$\ch(\gamma)=\ch(\gamma_1)+\ch(\gamma_2).$$
\end{lemma}

\begin{proof} It suffices to show that for all $1$-cells $\sigma$, $\pi_\sigma(\gamma)=\pi_\sigma(\gamma_1)+\pi_\sigma(\gamma_2)$. This is true since $q_\sigma\circ \gamma$ is the product of $q_\sigma\circ \gamma_1$ and $q_\sigma\circ \gamma_2$ in $\pi_1(Y^1/(Y^1-\sigma))$.
\end{proof}

The space of cellular $1$-chains on $Y$ is identified with $H_1(Y^1,Y^0)$. Under this identification, associated chains are simply the pushforward of a fundamental class.

\begin{lemma}\label{ch is pushforward of fundamental class} For an arc $\gamma$, the associated chain $\ch(\gamma)$ is the pushforward of the fundamental class of $[0,1]$ under the map $$\gamma_*:H_1([0,1],\partial [0,1])\to H_1(Y^1,Y^0).$$ For a closed curve $\gamma$, the associated chain $\ch(\gamma)$ is the pushforward of the fundamental class of $S^1$ under the composition $$H_1(S^1)\overset{\gamma_*}\to H_1(Y^1)\to H_1(Y^1,Y^0).$$
\end{lemma}

\begin{proof} There is an identification $$H_1([0,1],\partial [0,1])\simeq H_1(S^1)$$ and $$H_1(Y^1,Y^0)\simeq H_1(Y^1/Y^0)\simeq \mathbb{Z}^{|\text{1-cells of Y}|}.$$ The space $Y^1/Y^0$ is a wedge of circles. Let $$q:(Y^1,Y^0)\to (Y^1/Y^0,\{\text{pt}\})$$ be the quotient map. An curve $\gamma$ on $Y^1$, or an arc $\gamma$ on $Y^1$ with $\partial\gamma\subset Y^0$, is sent to a closed curve $q(\gamma)$ on $Y^1/Y^0$. 

In either case of whether $\gamma$ is an arc or a closed curve, $q(\gamma)$ is a closed curve on $Y^1/Y^0$. Viewing $q(\gamma):S^1\to Y^1/Y^0$ as a map, the pushforward of the relevant fundamental class in the lemma statement is the pushforward of the fundamental class of $S^1$ with respect to map induced by $q(\gamma)$ on homology. This gives a element $H_1(Y^1/Y^0)$.

Given a $1$-cell $\sigma$ of $Y$, the $\sigma$-coefficient of this element may be obtained by further quotienting to $Y^1/(Y^1-\sigma)$ and taking its image in homology. This recovers $\pi_\sigma(\gamma)$ from \cref{coefficients of ch}. 
\end{proof}

Associated chains are also natural with respect to cellular maps.

\begin{lemma}\label{ch is natural} If $f:Y\to Z$ is a cellular map and $\gamma:S^1\to Y^1$ is a closed curve, then $$f_*\ch(\gamma)=\ch(f(\gamma)).$$
\end{lemma}

\begin{proof} Since $f$ is cellular, $f:(Y_1,Y_0)\to (Z_1,Z_0)$. So we have $f_*:H_1(Y_1,Y_0)\to H_1(Z_1,Z_0)$. The lemma follows from \cref{ch is pushforward of fundamental class} and functoriality of homology.
\end{proof}

The homology class of a cellular chain associated to a closed curve $\gamma$ also agrees with the singular homology class $[\gamma]$ the curve determines. 

\begin{lemma}\label{hom of ch} Let $\gamma$ be a closed curve in $Y^1$. The cellular chain $\ch(\gamma)$ is a $1$-cycle, and the corresponding cellular homology class $[\ch(\gamma)]$ is identified with $[\gamma]$ under the isomorphism between cellular homology and singular homology.
\end{lemma}

\begin{proof} Inside $H_1(Y^1,Y^0)$, the space of cycles is canonically isomorphic to $H_1(Y^1)$. This isomorphism is induced by the inclusion map $j_*:H_1(Y^1)\to H_1(Y^1,Y^0)$ (see \cite[Theorem 2.35]{Hat02}). By \cref{ch is pushforward of fundamental class}, $\ch(\gamma)=j_*([\gamma])$ (in particular, it lies in $H_1(Y^1)$). 

To build the isomorphism $$H_1(Y)\simeq H_1^{\cell}(Y),$$ we have a long exact sequence
$$H_2(Y,Y^1)\overset{\delta}\to H_1(Y^1)\to H_1(Y)\to H_1(Y,Y^1)\simeq 0$$ which identifies $H_1(Y)\simeq H_1(Y^1)/\im(\delta)$. Because $\im(\delta)$ is identified with the boundaries in $H_1(Y^1)$, we obtain an isomorphism between $H_1(Y)$ and $H_1^{\cell}(Y)$. This isomorphism identifies the singular chain $[\gamma]$ with the cellular chain $j_*([\gamma])$.
\end{proof}

\subsection{Norms of cellular chains}

In this section, we define the $\ell^1$ and $\ell^\infty$ norms on spaces of cellular chains.

\begin{definition} Let $$\alpha=\sum_{\sigma \text{ is a }1\text{-cell in } Y}c_\sigma \sigma$$ be a cellular $2$-chain. The $\ell^1$-norm of $\alpha$ is $$\|\alpha\|_1=\sum_{\sigma \text{ is a }1\text{-cell in } Y}|c_\sigma|.$$
\end{definition}

\begin{lemma}\label{curves associated to cycle} Let $\alpha$ be a cellular $1$-cycle on $Y$. There exist embedded curves $$\gamma_1,...,\gamma_k:S^1\to Y^1$$ such that $$\alpha=\sum_{i=1}^k \ch(\gamma_i)$$ and $$\|\alpha\|_1= \sum_{i=1}^k \|\ch(\gamma_i)\|_1.$$
\end{lemma}

\begin{proof} We do induction on $\|\alpha\|_1$. Let $$\alpha=\sum_{\sigma \text{ is a }1\text{-cell in } Y}c_\sigma \sigma.$$ View $Y^1$ as a graph with the $0$-cells as vertices and $1$-cells as edges. Let $\sigma_1$ be a $1$-cell in $Y$ with $c_{\sigma_1}\neq 0$. Suppose $c_{\sigma_1}>0$. If $\partial \sigma_1=0$, then $\sigma_1$ is a closed curve. Reducing $|c_{\sigma_1}|$ by $1$, we have decreased $\|\alpha\|_1$. So assume $\partial \sigma_1=x_1-x_0$. Since $\partial \alpha=0$, there exists a $1$-cell $\sigma_2$ such that $$\partial \sign(c_{\sigma_2})\sigma_2=x_2-x_1$$ for some vertex $x_2$. If $x_2=x_0$, the edges $\sigma_1$ and $\sigma_2$ produce a closed curve, and we again eliminate the curve and reduce $\|\alpha\|_1$. Otherwise, since $\partial \alpha=0$, there exists a $1$-cell $\sigma_3$ such that $$\partial \sign(c_{\sigma_3})\sigma_3=x_3-x_2$$ for some vertex $x_3$. Continuing this process, we produce a string of $1$-cells whose endpoint vertex is a vertex already reached. This produces a closed curve $\gamma$ we may eliminate to reduce $\|\alpha\|_1$. By construction, $$\|\alpha\|_1=\|\alpha-\ch(\gamma)\|_1+\|\ch(\gamma)\|_1.$$ This completes the induction step.
\end{proof}

\begin{definition} For $p=1,2$, let $$\alpha=\sum_{\sigma \text{ is a }p\text{-cell in } Y}c_\sigma \sigma$$ be a cellular $p$-chain. The $\ell^\infty$-norm of $\alpha$ is $$\|\alpha\|_\infty=\sup_{\sigma \text{ is a }p\text{-cell in } Y} c_\sigma.$$
\end{definition}

\begin{definition} \label{def of N} Given a $1$-cell $\chi$, let $N(\chi)$ be the number of $2$-cells $\sigma$ so that the coefficient of $\sigma$ in $\partial \chi$ is nonzero. Let $$N(Y)=\sup_{\chi \text{ is a }1\text{-cell in } Y}N(\chi).$$ 
\end{definition}

\begin{lemma}\label{boundary of k embedded chain} Let $\alpha$ be a cellular $2$-chain on $Y$. Then $$\|\partial \alpha\|_\infty\leq N(Y)\|\alpha\|_\infty.$$
\end{lemma}

\begin{proof} We must show that the coefficient of any $1$-cell $\chi$ in the expansion of $\partial \alpha$ is at most $N(Y)k$. Writing 
\begin{align*}\partial\alpha&=\partial\left(\sum_{\sigma \text{ is a }2\text{-cell in } Y}c_\sigma \sigma\right)\\&=\sum_{\sigma \text{ is a }2\text{-cell in } Y}c_\sigma \partial \sigma,
\end{align*} the only contributions to the coefficient of $\chi$ in $\partial \alpha$ come from terms associated to $\sigma$ such that the $\chi$-coefficient in $\partial \sigma$ is nonzero. By assumption, there are at most $N(Y)$ such $\sigma$. Since $|c_\sigma|\leq k$, the lemma follows.
\end{proof}

\subsection{Isoperimetric inequality on $X^2$} Let $X$ be the CW complex associated to a cubical lattice in $\mathbb{R}^3$. Equip $X$ with the standard CW structure, and let $X^2$ and $X^1$ be the associated $2$-skeleton and $1$-skeleton, respectively. We first prove an isoperimetric inequality for curves on $X^1$.

\begin{lemma} \label{isoperimetric inequality on X^1} Let $\alpha$ be a cellular $1$-cycle on $X$. There exists a cellular $2$-chain $C$ on $X$ such that $\partial C=\alpha$ and $\|C\|_\infty\leq \|\alpha\|_1$.
\end{lemma}

\begin{proof} By \cref{curves associated to cycle}, we reduce to the case where $\alpha=\ch(\omega)$ for an embedded closed curve $\omega$ on $X^1$. View the $1$-complex $X^1$ as a graph. We prove a slightly more general statement. Let $\omega$ be a locally embedded closed curve on $X^1$, meaning that $\omega$ is the curve associated to a cycle of adjacent edges of $X^1$. Let $|\omega|$ be the number of edges in the cycle associated to $\omega$. We fill $\ch(\omega)$ with a $2$-chain $C$ such that $\|C\|_\infty\leq |\omega|$. (When $\omega$ is actually embedded, $|\omega|=\|\ch(\omega)\|_1$, so the lemma follows.) To do this, we will do induction on $|\omega|$.  

Each (undirected) edge of the graph $X^1$ is in either the $x$, $y$ or $z$ direction. An (oriented) locally embedded curve $\omega$ gives a word in $s_\omega$ in six letters: $x$, $y$, $z$, $x^{-1}$, $y^{-1}$ and $z^{-1}$. Adjacent pairs of $x$ and $x^{-1}$, $y$ and $y^{-1}$, or $z$ and $z^{-1}$ may be eliminated since they cancel in $\ch(\omega)$. A transposition in the word corresponds to moving an arc of $\omega$ across a square cell of $X^2$. For example, a transposition between $x$ and $y$ corresponds to moving an arc of $\omega$ across a square cell of $X^2$ in an $x,y$-plane. 

Choose an innermost pair of $x$ and $x^{-1}$, $y$ and $y^{-1}$, or $z$ and $z^{-1}$. For example, suppose an innermost pair is a pair of $x$ and $x^{-1}$. We may do a number of transpositions of $x$ so that it is adjacent to $x^{-1}$, and then eliminate the pair. The substring of $s_\omega$ between $x$ and $x^{-1}$ is monotonic in $y$ and $z$ (i.e. it contains only one of $y$ and $y^{-1}$ and only one of $z$ and $z^{-1}$), by our innermost pair assumption. The cycle associated to this substring contains any edge at most once. Therefore, the squares corresponding to the transpositions are all distinct. A similar argument eliminates an innermost pair of $y$ and $y^{-1}$ or $z$ and $z^{-1}$. So we produced a new locally embedded curve $\omega'$ with $|\omega'|=|\omega|-1$, and a new word $s_{\omega'}$. We also produced a cellular $2$-chain $C'$ such that $$\partial C'=\ch(\omega)-\ch(\omega')$$ and $\|C'\|_\infty\leq 1$. By induction hypothesis, $\ch(\omega')$ may be filled by a $2$-chain $C''$ with $$\|C''\|_\infty\leq |\omega'|.$$ Taking $C=C'+C''$, we see that $\partial C=\ch(\omega)$ and 
\begin{align*}\|C\|_\infty &\leq \|C'\|_\infty+\|C''\|_\infty \\&\leq|\omega'|+1\\&\leq |\omega|.
\end{align*} This completes the induction step.
\end{proof}

Let $\widetilde{X}$ be a new CW $2$-complex whose total space is $X^2$ and CW structure is a refinement of the CW structure on $X^2$ coming from the CW structure on $X$. Denote by $\widetilde{X}^1$ the $1$-skeleton of this CW complex.

\begin{lemma}\label{isoperimetric inequality on X^2} Let $\omega$ be an embedded closed curve on $\widetilde{X}^1$. Then there exists a cellular $2$-chain $C$ on $\widetilde{X}$ such that $\partial C=\ch(\omega)$ and $\|C\|_\infty\leq 5(|\omega\cap X^1|+1)$.
\end{lemma}

\begin{proof} By assumption, $i:X^2\to \widetilde{X}^2$ is a cellular map. First, we to push $\omega$ into $X^1$. That is, we construct a cellular $1$-cycle $\alpha$ on $X$ and a $2$-chain $C'$ on $\widetilde{X}$ such that $$\|C'\|_\infty\leq |\omega\cap X^1|+1,$$ $$\partial C'=\ch(\omega)-i_*\alpha$$ and $$|\alpha|\leq 4(|\omega\cap X^1|+1).$$ 

To do this, since $\omega$ is an embedded curve on $\widetilde{X}^1$, $\omega$ is the concatenation of arcs $\omega_1,...,\omega_\ell$, each contained in a $2$-cell (square) of $X$, with boundary on $X^1$. (In the case where $\omega$ does not intersect $X^1$, we have only one arc which is a closed curve with no boundary, contained in a $2$-cell of $X$.) In either case, the number of arcs is bounded: $$\ell\leq |\omega\cap X^1|+1.$$ By \cref{ch concatenation}, $$\ch(\omega)=\sum_{i=1}^{\ell}{\ch(\omega_i)}.$$ For each $\omega_i$ contained in a $2$-cell $Q_i$ of $X$, there is an arc $\delta_i\subset \partial Q_i$ and an embedded disk on $Q_i$ with boundary $\omega_i\cup \delta_i$. This embedded disk is the union of the closure of some distinct $2$-cells of $\widetilde{X}$. Hence, there is a $2$-chain $C'_i$ such that $\partial C'_i=\ch(\omega_i)-\ch(\delta_i)$ and $\|C'_i\|_\infty\leq 1$. Let $\delta\subset X^1$ be the closed curve formed by concatenating the $\delta_i$s. By \cref{ch concatenation} and \cref{ch is natural}, \begin{align*}\sum_{i=1}^{\ell}\ch(\delta_i)&=\ch(i_*\delta)\\&=i_*\ch(\delta).
\end{align*} So $$\partial \left(\sum_{i=1}^{\ell} C'_i\right)=\ch(\omega)-i_*\ch(\delta).$$

The curve $\delta$ is the union of $\ell$ embedded arcs $\delta_i\subset \partial Q_i$. Given a $1$-cell $\sigma$ of $X$, $$|\pi_\sigma(\delta)|\leq |\{1\leq i\leq \ell|\partial Q_i \text { has a nonzero coefficient of }\sigma\}|.$$ Hence 
\begin{align*}|\ch(\delta)|&\leq 4\ell\\&\leq 4(|\omega\cap X^1|+1).
\end{align*}

Let $$C'=\sum_{i=1}^{\ell}C'_i.$$ Since $\|C'_i\|_\infty\leq 1$ for all $i$, 
\begin{align*}\|C'\|_\infty&\leq \ell\\ &\leq  |\omega\cap X^1|+1.
\end{align*} Taking $$\alpha=\ch(\delta),$$ we see that $|\alpha|\leq 4(|\omega\cap X^1|+1)$ and $$\partial C'=\ch(\omega)-i_*\alpha$$ as desired. 

By \cref{isoperimetric inequality on X^1}, there exists a cellular $2$-chain $C''$ on $X$ with $$\|C''\|_\infty\leq 4(|\omega\cap X^1|+1)$$ such that $\partial C''=\alpha$. Because $i:X^2\to \widetilde{X}^2$ is the inclusion map, $$\|i_*C''\|_\infty\leq  4(|\omega\cap X^1|+1)$$ also. Taking $C=i_*C''+C'$, we note that $C$ is a cellular $2$-chain on $\widetilde{X}$ satisfying \begin{align*}\|C\|_\infty&\leq \|i_*C''\|_\infty+\|C'\|_\infty\\&\leq 5(|\omega\cap X^1|+1)
\end{align*} and $\partial C=\ch(\omega)$.
\end{proof}

\section{Proof of \cref{main embedding theorem CW version}}\label{proof of CW intersection}

Let $T\subset X\simeq \mathbb{R}^3$ be an embedded torus, in general position with the CW structure of $X$. Let $u$ and $v$ be primitive homology classes in $H_1(T)$ that are trivial in the two components of $\mathbb{R}^3-T$, respectively. Suppose $$\textstyle\sys_u(T)\geq 20\cdot(|X^1\cap T|+1)\cdot\length(X^2\cap T).$$ Assume that no $3$-cell of $X$ contains a representative of $v$ on $T$. We will show a contradiction. 

We may assume that $u$ and $v$ are primitive homology classes in $H_1(T)$. First, we prove lemmas to show that $X^2\cap T$ contains a representative of some $au+bv$, with $a$ nonzero.

\begin{lemma} \label{1-complex on torus} Let $J$ be a $1$-complex on $T$. If $T- J$ does not contain any representative of $v$, then $J$ contains a representative of $$au+bv\in H_1(T)$$ with $a\neq 0$.
\end{lemma}

\begin{proof} We will show the contrapositive. Suppose $J$ does not contain any representative of $au+bv$ with $a\neq 0$. Then, the image of the map $H_1(J)\to H_1(T)$ is contained in the subgroup of $H_1(T)$ generated by $v$. We divide into two cases.

\textbf{Case 1.} The image is trivial. Let $U$ be a sufficiently small neighborhood of $J$, so that $U$ deformation retracts onto $J$. Now, $\partial U$ is a disjoint union of simple closed curves on $T$, and all these curves must be inessential. Therefore, either $U$ contains representatives of $u$ and $v$, or $T-U$ contains representatives of $u$ and $v$. The first situation cannot happen because the image of $H_1(U)\to H_1(T)$ is also trivial. Therefore $T- U$ (and thus also $T- J$ contains a representative of $v$). Thus, the lemma is true in this case.

\textbf{Case 2.} The image is nontrivial. So for some $b\neq 0$, a representative of $bv$ is contained in $J$. In this situation, $J$ must contain a representative of $v$ also. Now, suppose the lemma statement is false in this case, meaning $T-J$ contains no representative of $v$. This means $T-J$ contains no representative of $bv$ either, for $b\neq 0$. Since the intersection number of $v$ and $au+bv$ is $a$ and $J$ contains a representative of $v$, $T-J$ does not contain any representative of $au+bv$ with $a\neq 0$. So, the image of $H_1(T- J)\to H_1(T)$ is trivial. Again, let $U$ be a sufficiently small neighborhood of $J$ so that $U$ deformation retracts onto $J$. Then $\partial U$ is a disjoint union of simple closed curves on $T$, not intersecting $J$. Since the image of $H_1(T-J)\to H_1(T)$ is trivial, all the curves in $\partial U$ are inessential. This means either $U$ (and thus $J$) contains representatives of $u$ and $v$, or $T- U$ (and thus $T- J$) contains representatives of $u$ and $v$. Both situations lead to contradictions.
\end{proof}

We next show that we may choose the representative in \cref{1-complex on torus} to be embedded. 

\begin{lemma} \label{1-complex on torus embedded version} Let $J$ be a $1$-complex on $T$. Suppose $J$ contains a representative of $$au+bv,$$ with $a\neq 0$. Then $J$ contains an embedded representative of $$a'u+b'v$$ with $a'\neq 0$. 
\end{lemma}

\begin{proof} Let $\alpha$ be the representative of $au+bv$ in the hypothesis of the lemma. If $\alpha$ is not embedded, then $\alpha$ decomposes as $\alpha_1\cup \alpha_2$ for closed curves $\alpha_1$ and $\alpha_2$ on $T$. The homology classes of $\alpha_1$ and $\alpha_2$ may be expressed as linear combinations of $u$ and $v$. Since $\alpha$ contains a nonzero factor of $u$, at least one of $\alpha_1$ and $\alpha_2$ contains a nonzero factor of $u$. Now, we continue this process of decomposition until we obtain an embedded curve whose homology class contains a nonzero factor of $u$.
\end{proof}

We continue with the proof \cref{main embedding theorem CW version}. Assuming it is false, there is no $3$-cell of $X$ containing a representative of $v$ on $T$. Let $J$ be the $1$-complex $T\cap X^2$ on $T$. Our supposition implies that $T-J$ does not contain any representative of $v$. By \cref{1-complex on torus}, $J$ does contain a representative of $au+bv$, with $a\neq 0$. By \cref{1-complex on torus embedded version}, $J$ contains an embedded representative $\omega$ of $a'u+b'v$, with $a'\neq 0$.

The curve $\omega$ lies in $T\cap X^2$, by construction. Viewing $\omega$ as an embedded curve on $X^2$, we have $$|\omega\cap X^1|\leq |T\cap X^1|.$$ Put a CW structure $\widetilde{X}$ on $\mathbb{R}^3$ such that the inclusions $X\to \widetilde{X}$ and $T\to \widetilde{X}$ are cellular maps. Pull back the CW structure $\widetilde{X}$ to obtain CW structures $\widetilde{X^2}$ and $\widetilde{T}$ on $X^2$ and $T$, respectively. The CW structure $\widetilde{X^2}$ is a refinement of the CW structure on $X^2$ viewing it as the $2$-skeleton of $X$. The embedded curve $\omega$ lies in the $1$-skeleton of $\widetilde{X}$ (also the $1$-skeleton of $\widetilde{X^2}$). By \cref{isoperimetric inequality on X^2} applied to the $2$-complex $\widetilde{X^2}$, there is a $2$-chain $C$ on $\widetilde{X^2}$ such that $$\|C\|_\infty\leq 5(|T\cap X^1|+1)$$ and $$\partial C=\ch(\omega).$$ 

Let us label the closures of the two connected components of $\mathbb{R}^3-T$ as $T_u$ and $T_v$, so that the homology class $u$ is trivial in $T_u$, and $v$ is trivial in $T_v$. Note that $$H_1(T_u)\simeq H_1(T_v)\simeq \mathbb{Z},$$ generated by $v$ and $u$, respectively. Pulling back the CW structure on $\widetilde{X}$ gives CW structures $\widetilde{T_u}$ and $\widetilde{T_v}$ on $T_u$ and $T_v$, respectively. Each (open) $2$-cell of $\widetilde{X}$ is contained in one of the components $T_u$ or $T_v$ (or both, if it is in $T$). Since $C$ is a cellular chain on $\widetilde{X^2}$, each cell of $C$ is contained in one of the $T_u$ or $T_v$. Thus, $$C=C_{u}+C_{v},$$ where $C_u$ consists of the $2$-cells of $C$ lying in $T_u$, and $C_v$ consists of the $2$-cells of $C$ lying in $T_v$. 

Note that $\ch(\omega)$ is a cellular $1$-cycle on $\widetilde{T}$. Since $\omega$ represents the singular homology class $a'u+b'v$ in $H_1(T)$, by \cref{hom of ch}, $\ch(\omega)$ represents the cellular homology class $a'u+b'v$ in $H_1^{\cell}(\widetilde{T})$. Now, $\partial C_u$ is also a cellular $1$-chain on $\widetilde{T}$. It is a cycle on $\widetilde{T}$ since it is a cycle (actually a boundary) on $\widetilde{T_u}$. Furthermore, $\partial C_u$ represents the homology class $a'u$ in $H_1^{\cell}(\widetilde{T})$, where $a'\neq 0$. 

Because $\widetilde{X^2}$ is a refinement of the standard CW structure on $X^2$ viewing it as a cubical lattice, $$N(\widetilde{X^2})=4$$ (see \cref{def of N}). Since 
\begin{align*}\|C_{u}\|_\infty&\leq\|C\|_\infty\\ &\leq 5(|T\cap X^1|+1),
\end{align*} \cref{boundary of k embedded chain} implies that  $$\|\partial C_{u}\|_\infty\leq 20(|T\cap X^1|+1).$$ 

By \cref{curves associated to cycle}, there exists a closed curve $\gamma$ (possibly with multiple components) on the $1$-skeleton of $\widetilde{T}$ with multiplicity at most $20(|T\cap X^1|+1)$ such that $$\ch(\gamma)=\partial C_u.$$ Hence, $$\length (\gamma)\leq 20\cdot (|T\cap X^1|+1)\cdot \length(T\cap X^2).$$ By \cref{hom of ch}, $\gamma$ represents the singular homology class $a'u$. If $|a'|=1$, this contradicts the systole assumption. If $|a'|\neq 1$, there is a subset of $\gamma$ which is a curve representing $u$, which also contradicts the systole assumption, completing the proof of \cref{main embedding theorem CW version}.

\appendix

\section{Some extensions}

\subsection{A sharp example for \cref{main embedding theorem}} \label{Appendix A1} In this section, we describe an example due to Aleksandr Berdnikov, which shows that \cref{main embedding theorem} is sharp as $\sys_u(T)\to \infty$.

Let $\beta$ be a torus knot $T_{n,n-1}$ of length $\sim n$ lying on a standard torus $$\{(x,y,z)\in \mathbb{R}^3|(\sqrt{x^2+y^2}-2)^2+z^2=1\}$$ in $\mathbb{R}^3$. Let $U$ be the unit normal disk bundle of the torus restricted to $\beta$, so that $U$ is a knotted annulus in $\mathbb{R}^3$. Push $U$ infinitesimally along its normal bundle, keeping $\partial U$ fixed, to get an isotopic copy $U'$.  The interior of $U'$ does not intersect $U$, and  $\partial U'=\partial U$. Let $T=U\cup U'$, so that $T$ is a knotted torus embedded in $\mathbb{R}^3$. Let $\zeta_1$ and $\zeta_2$ be the two connected components of $\partial U$. They are isotopic curves on $T$.

Let the interior of $T$ (denoted $\interior(T)$) be the bounded connected component of $\mathbb{R}^3-T$. Let $v$ be the primitive homology class that is trivial in $\interior(T)$. Then $v$ is represented by a curve $\gamma$ lying in an infinitesimal neighborhood of a unit normal vector in $U$. In particular, $\diam_{\mathbb{R}^3}(\gamma)\sim 1$.

Let the exterior of $T$ (denoted $\exterior(T)$) be the unbounded connected component of $\mathbb{R}^3-T$. Let $u$ be the primitive homology class of $\mathbb{R}^3-T$ that is trivial in $\exterior(T)$. We construct $u$ explicitly as follows. Let $w$ be the homology class in $H_1(T)$ represented by $\zeta_1$ and $\zeta_2$. Note that $v$ and $w$ are linearly independent. Now, we express $u$ in terms of $v$ and $w$. Let $D\subset \mathbb{R}^3$ be an immersed disk with $\partial D=\zeta_1$. Isotope $D$ so that it intersects $T$ and $\zeta_2$ transversely. Then $D\cap \exterior(T)$ and $D\cap \interior(T)$ are immersed submanifolds in $\mathbb{R}^3$. The boundary of $D\cap \exterior(T)$ represents a homology class $a\in H_1(T)$, while the boundary of $D\cap \interior(T)$ represents a homology class $b\in H_1(T)$. By construction, $a+b=w$. Moreover, $b$ is a multiple of $v$ since it must be trivial in the homology of $\interior(T)$. The multiplicity of $v$ is the linking number of $\zeta_1$ and $\zeta_2$, which is $n(n-1)$. So $w=a+n(n-1)v$. By construction, $a$ must be a multiple of $u$. Furthermore, $a=w-n(n-1)v$ is primitive. Thus, $$u=a=w-n(n-1)v.$$ Now, $\sys_w(T)\sim \length(\beta)\sim n$, while $\sys_v(T)\sim 1$. Hence, $\sys_u(T)\sim n(n-1)$.

On the other hand, $\Area(T)\sim \length(\beta)\sim n$. Therefore, in this example, $$\textstyle\sys_{\mathbb{R}^3}(T)\sim \sys_u(T)^{-1/3}\Area(T)^{2/3}.$$

\subsection{Another extrinsic systolic inequality} \label{Appendix A2} In this section, we include a proof of the following result. 

\begin{theorem}\label{Pardon embedding theorem} Let $T\subset \mathbb{R}^3$ be an embedded torus. Let $u$ and $v$ be homology classes on $H_1(T)$ that are trivial in the homology of each component of $\mathbb{R}^3-T$, respectively. Then $$\textstyle\sys_{\mathbb{R}^3}(T)\lesssim \displaystyle\frac{\Area(T)}{\min\{\sys_u(T),\sys_v(T)\}}.$$
\end{theorem}

The proof is very similar to the proof of the main theorem in \cite{Par11}. The topological part of the argument is the same, but the geometric quantities measured are different.

\begin{proof}[Proof of \cref{Pardon embedding theorem}]
Perturbing, we may assume $T$ is smooth. Scaling, we may assume that $\Area(T)=1$. Denote by box of scale $r$ a box with dimensions $r$, $2^{1/3}r$ and $2^{2/3}r$. We would like to consider the box of the smallest scale containing representatives of both $u$ and $v$ on $T$. To avoid transversality and other technical issues, we simply let $r_0>0$ be a value such that there exists a box $Q$ of scale $r_0$ containing representatives of both $u$ and $v$ on $T$, but no box of scale $(2/3)r_0$ contains representatives of both $u$ and $v$ on $T$. We also assume that $Q$ intersects $T$ transversely. By construction, $$r_0\geq \textstyle \sys_{\mathbb{R}^3} (T).$$ Assume that $$r_0>20\textstyle(\min\{\sys_u(T),\sys_v(T)\})^{-1}.$$ We will prove a contradiction.

Translate $Q$ and $T$ so that $Q$ is centered at the origin $(0,0,0)$. For $r\in [r_0,6r_0/5]$, let $Q_r$ be the box of scale $r$ centered at the origin. By our area normalization, $$\int_{r_0}^{6r_0/5}\length(\partial Q_{r}\cap T)dr\leq 1.$$ So for some $Q'$ among the $Q_r$s, $$\length(\partial Q'\cap T)\leq 5r_0^{-1}.$$ Again, we may assume that $Q'$ intersects $T$ transversely.

Next, we find a plane intersecting the long dimension of $Q'$ whose intersection with $T$ is bounded. Without loss of generality, assume $z$ is the long dimension of $Q'$. Let $P_s$ be the intersection of the plane $z=s$ with the $Q'$. Again by our area normalization,
$$ \int_{-r_0/10}^{r_0/10}\length(P_s\cap T)ds\leq 1.$$ Therefore, for some $s\in [-r_0/10,r_0/10]$, $$|P_s\cap \beta|\leq 5r_0^{-1}.$$ As always, we choose $P_s$ so that it intersects $T$ transversely.

The plane $P_s$ divides $Q'$ into two boxes, which we label $Q'_1$ and $Q'_2$. Apriori, $Q'_1$ and $Q'_2$ may not be boxes of any scale, as they may have the wrong ratio of dimensions. However, both $Q'_1$ and $Q'_2$ are contained in boxes of scale $(2/3)r_0$. Since $Q'$ contains $Q$, it contains representatives of both $u$ and $v$ in $H_1(T)$. Thus, $$\genus(Q'\cap T)=1.$$ Since $Q$ was assumed to be the smallest box containing representatives of $u$ and $v$, $$\genus(Q'_1\cap T)=\genus(Q'_2\cap T)=0.$$

Let $S_t$ be the $1$-parameter family of spheres, starting from $S_0=Q'$ and doing surgery along $P_s$, so that $S_1$ is a disjoint union of spheres. Since $$\length((Q'\cup P_s)\cap T)\leq 10r_0^{-1},$$ we may assume that $$\length(S_t\cap T)\leq 20r_0^{-1}$$ for all $0\leq t\leq 1$. Since $$r_0^{-1}< \frac{\textstyle\min\{\sys_u(T),\sys_v(T)\}}{20},$$ $S_t$ only intersects $T$ in inessential curves. (If it intersected $T$ in an essential curve, we could find an embedded representative of $u$ or $v$, which would give a contradiction to the length bound on $S_t\cap T$.) So we have a $1$-parameter family of spheres $S_t$, each intersecting $T$ only in inessential curves, such that $$\genus(S_0\cap T)=\genus(Q'\cap T)=1$$ and $$\genus(S_1\cap T)=\genus(Q'_1\cap T)+\genus(Q'_2\cap T)=0.$$ This is a contradiction by \cite[Lemma 2.3]{Par11}. 
\end{proof}

\subsection{Twisted embeddings of flat tori} \label{Appendix A3} In this section, we describe the following question of Guth, which is a generalization of \cref{twisted cylinder}. 

Let $\mathbb{R}^2/\mathbb{Z}^2$ be the flat torus equipped with the Euclidean metric. Let $$f:\mathbb{R}^2/\mathbb{Z}^2\to T\subset \mathbb{R}^3$$ be an unknotted $C^1$-isometric embedding. There is an induced map on homology given by $$f_*:H_1(\mathbb{R}^2/\mathbb{Z}^2)\to H_1(T).$$ There is a canonical identification $H_1(\mathbb{R}^2/\mathbb{Z}^2)\simeq \mathbb{Z}^2$, as follows. A point in $\mathbb{Z}^2\subset \mathbb{R}^2$ determines a line to the origin, which in turn gives a closed curve and homology class on the torus. Let $u$ and $v$ in $H_1(T)$ be homology classes trivial in the homology of the two components of $\mathbb{R}^3-T$, respectively. Then $u$ and $v$ also give an identification $H_1(T)\simeq \mathbb{Z}^2$. Under these identifications, $f_*\in SL_2(\mathbb{Z})$ is a matrix that encodes the topology of the embedding of the flat torus.

\begin{question} What is the best upper bound for $\sys_{\mathbb{R}^3}(T)$ in terms of the matrix $f_*$?
\end{question}

To study this question, it helps to consider the matrix entries $$f_*=\begin{pmatrix}
a & b\\
c& d
\end{pmatrix}\in SL_2(\mathbb{Z}).$$ This means, \begin{equation*}
\begin{cases} f_*^{-1}u=(d,-b)\\
f_*^{-1}v=(-c,a).
\end{cases}
\end{equation*}
Hence, $\sys_u(T)=(b^2+d^2)^{1/2}$ and $\sys_v(T)=(a^2+c^2)^{1/2}$. \cref{main embedding theorem} implies $$\textstyle\sys_{\mathbb{R}^3}(T)\lesssim \min\{(b^2+d^2)^{-1/6},(a^2+c^2)^{-1/6}\}$$ while \cref{Pardon embedding theorem} implies $$\textstyle\sys_{\mathbb{R}^3}(T)\lesssim \max\{(b^2+d^2)^{-1/2},(a^2+c^2)^{-1/2}\}.$$ If $b^2+d^2\lesssim (a^2+c^2)^{1/3}$ (or $a^2+c^2\lesssim (b^2+d^2)^{1/3}$), then \cref{main embedding theorem} gives a better bound. Otherwise, \cref{Pardon embedding theorem} gives a better bound.   

Examples can be constructed using the Nash isometric embedding theorem. There is a $1$-Lipschitz map from $\mathbb{R}^2/\mathbb{Z}^2$ to the standard torus in $\mathbb{R}^3$ scaled by $\min\{(a^2+b^2)^{-1/2},(c^2+d^2)^{-1/2}\}$, whose topology is given by $f_*$. By the Nash isometric embedding theorem, we may perturb the embedding locally to make it a $C^1$-isometry to a torus $T\subset \mathbb{R}^3$, so that $$\textstyle\sys_{\mathbb{R}^3}(T)\sim \min\{(a^2+b^2)^{-1/2},(c^2+d^2)^{-1/2}\}.$$ When $b^2+d^2\sim a^2+c^2$, the upper bound given by \cref{Pardon embedding theorem} matches the lower bound above asymptotically. For other ranges of $a$, $b$, $c$ and $d$, there is a gap between the upper and lower bounds. 

\bibliographystyle{alpha} 

\bibliography{bibliography}

\end{document}

%% file: Fig2.pdf_tex
%% Creator: Inkscape 1.1.2 (0a00cf5339, 2022-02-04), www.inkscape.org
%% Accompanies image file 'Fig2.pdf' (pdf, eps, ps)
\begingroup%
  \makeatletter%
  \providecommand\color[2][]{%
    \errmessage{(Inkscape) Color is used for the text in Inkscape, but the package 'color.sty' is not loaded}%
    \renewcommand\color[2][]{}%
  }%
  \providecommand\transparent[1]{%
    \errmessage{(Inkscape) Transparency is used (non-zero) for the text in Inkscape, but the package 'transparent.sty' is not loaded}%
    \renewcommand\transparent[1]{}%
  }%
  \providecommand\rotatebox[2]{#2}%
  \newcommand*\fsize{\dimexpr\f@size pt\relax}%
  \newcommand*\lineheight[1]{\fontsize{\fsize}{#1\fsize}\selectfont}%
  \ifx\svgwidth\undefined%
    \setlength{\unitlength}{124.10734657bp}%
    \ifx\svgscale\undefined%
      \relax%
    \else%
      \setlength{\unitlength}{\unitlength * \real{\svgscale}}%
    \fi%
  \else%
    \setlength{\unitlength}{\svgwidth}%
  \fi%
  \global\let\svgwidth\undefined%
  \global\let\svgscale\undefined%
  \makeatother%
  \begin{picture}(1,0.69485073)%
    \lineheight{1}%
    \setlength\tabcolsep{0pt}%
    \put(0,0){\includegraphics[width=\unitlength,page=1]{Fig2.pdf}}%
    \put(0.56874031,0.18052652){\makebox(0,0)[lt]{\lineheight{1.25}\smash{\begin{tabular}[t]{l}$\scriptstyle{\times}\displaystyle{n}$\end{tabular}}}}%
  \end{picture}%
\endgroup%

%% file: Twisted_Embeddings_of_Tori.bbl
\begin{thebibliography}{BPS12}

\bibitem[Bav86]{Bav86}
C.~Bavard.
\newblock In\'{e}galit\'{e}s isosystoliques pour la bouteille de {K}lein.
\newblock {\em Math. Ann.}, 274(3):439--441, 1986.

\bibitem[BK03]{BK03}
V.~Bangert and M.~Katz.
\newblock Stable systolic inequalities and cohomology products.
\newblock {\em Comm. Pure Appl. Math}, 56:979--997, 2003.

\bibitem[BK04]{BK04}
V.~Bangert and M.~Katz.
\newblock An optimal {L}oewner-type systolic inequality and harmonic one-froms
  of constant norm.
\newblock {\em Comm. Anal. Geom.}, 12(3):703--732, 2004.

\bibitem[Bla61]{Bla61}
C.~Blatter.
\newblock Zur {R}iemannschen geometrie im grossen auf dem mobiusband.
\newblock {\em Compositio Math.}, 15:88--107, 1961.

\bibitem[BPS12]{BPS12}
F.~Balacheff, H.~Parlier, and S.~Sabourau.
\newblock Short loop decompositions of surfaces and the geometry of
  {J}acobians.
\newblock {\em Geometric and Functional Analysis}, 22(1):37--73, 2012.

\bibitem[BS10]{BS10}
F.~Balacheff and S.~Sabourau.
\newblock Diastolic and isoperimetric inequalities on surfaces.
\newblock {\em Annales Scientifiques de l'École Normale Supérieure},
  43(4):579--605, 2010.

\bibitem[Gro83]{Gro83}
M.~Gromov.
\newblock Filling {R}iemannian manifolds.
\newblock {\em J. Differential Geom.}, 18(1):1--147, 1983.

\bibitem[Gro96]{Gro96}
M.~Gromov.
\newblock Systoles and intersystolic inequalities.
\newblock In {\em Actes de la Table Ronde de Geometrie Differentielle (Luminy,
  1992)}, Semin. Congr. 1, pages 291--362. Soc. Math. France, Paris, 1996.

\bibitem[Gro99]{Gro99}
M.~Gromov.
\newblock {\em Metric structures for {R}iemannian and non-{R}iemannian spaces}.
\newblock Birkauser, Boston, MA, 1999.

\bibitem[Gut10]{Gut10}
L.~Guth.
\newblock Metaphors in systolic geometry.
\newblock In {\em Proceedings of the ICM 2010}, pages 745--768, 2010.

\bibitem[Hat02]{Hat02}
A.~Hatcher.
\newblock {\em Algebraic {T}opology}.
\newblock Cambridge University Press, 2002.

\bibitem[IK04]{IK04}
S.~Ivanonv and M.~Katz.
\newblock Generaized degree and optimal loewner-type inequalities.
\newblock {\em Israel Journal of Math.}, 141:221--233, 2004.

\bibitem[Kat07]{Kat07}
M.~Katz.
\newblock {\em Systolic geometry and topology}.
\newblock Mathematical surveys and monographs v. 137. American Mathematical
  Society, Providence, RI, 2007.

\bibitem[Par11]{Par11}
J.~Pardon.
\newblock On the distortion of knots on embedded surfaces.
\newblock {\em Annals of Mathematics}, pages 637--646, 2011.

\bibitem[Pu52]{Pu52}
P.~Pu.
\newblock Some inequalities in certain nonorientable {R}iemannian manifolds.
\newblock {\em Pacific Journal of Mathematics}, 2(1):55--71, 1952.

\bibitem[Vas25]{Vas25}
S.~Vasudevan.
\newblock Extrinsic systole of {S}eifert surfaces and distortion of knots.
\newblock {\em arXiv}, 2025.

\end{thebibliography}
